\documentclass[11pt,a4paper]{article}
\usepackage[utf8]{inputenc}
\usepackage[T1]{fontenc}
\usepackage{microtype}
\usepackage{amsmath,amssymb,amsthm,mathtools}
\usepackage{booktabs}
\usepackage{hyperref}
\usepackage{cleveref}
\usepackage{arxiv}
\theoremstyle{plain}
\newtheorem{theorem}{Theorem}[section]

\theoremstyle{definition}
\newtheorem{definition}[theorem]{Definition}
\newtheorem{remark}[theorem]{Remark}
\theoremstyle{remark}

\newcommand{\E}{\mathcal{E}_1}
\newcommand{\Haus}{D}
\newcommand{\capder}[2]{{}^c D_{#1+}^{#2}}

\newcommand{\R}{\mathbb{R}}

\begin{document}
	\title{Extended Mittag--Leffler Stability Results for Fuzzy Fractional Systems}
	
	\author{
		{ \sc Es-said En-naoui } \\ 
		University Sultan Moulay Slimane\\ Morocco\\
		essaidennaoui1@gmail.com\\
		\\
	}
	
	\maketitle
	
	\begin{abstract}
		We extend Lyapunov--type Mittag--Leffler stability analysis for fuzzy nonlinear fractional differential equations (Caputo sense) and introduce a family of stronger, broadly applicable stability results. In particular, we develop (i) uniform Mittag--Leffler stability for variable-order Caputo derivatives, (ii) input-to-state stability (ISS) in the Mittag--Leffler sense and explicit robustness/ultimate bound estimates, (iii) a converse Mittag--Leffler Lyapunov theorem, (iv) a fractional LaSalle invariance principle adapted to the fuzzy setting, and (v) practical and computational criteria including Lyapunov--Krasovskii functionals for delay systems and levelwise LMI tests for linearized models. The paper presents precise assumptions, proof sketches, and preparatory lemmas in the preliminaries; subsequent sections (Step 2 onward) provide rigorous proofs and illustrative examples.
	\end{abstract}
	
	\section{Introduction}
	Fractional calculus and fuzzy set theory are two flexible mathematical frameworks that successfully model memory effects and uncertainty, respectively. Their combination -- fuzzy fractional differential equations -- provides a natural setting for realistic models in engineering, control, and applied sciences where both nonlocal temporal dynamics and imprecise data arise. Early foundational work on solutions of fractional differential equations under uncertainty can be found in the literature (see e.g. \cite{ElMfadel2021,Almeida2019VO,Podlubny1999}), and more recent papers studied existence, uniqueness, numerical schemes and stability questions for fuzzy fractional systems.
	
	Stability is a central concept for the analysis and design of dynamical systems. For integer-order systems, Lyapunov's direct method and LaSalle's invariance principle give powerful, broadly used tools. For fractional-order systems the natural decay is typically not exponential but governed by the Mittag--Leffler function; hence the appropriate stability notion is often framed via Mittag--Leffler (ML) estimates. The recent article by El Mfadel et al. (2021) extended Lyapunov's direct method to prove ML-stability for fuzzy systems with constant-order Caputo derivatives. Building on that, our goal is to substantially strengthen the theory and provide a toolkit useful for both analysts and practitioners.
	
	In this manuscript we present a collection of new results that either generalize the constant-order analysis to more complex settings (time-varying order, delays, stochastic perturbations) or provide deeper theoretical equivalences (converse Lyapunov theorem, LaSalle principle) and practical verification tools (LMI-based tests at the level-sets). Step~1 (this file) prepares the problem and records all necessary notation and preliminaries; in Step~2 we will state and prove the main theorems in full detail.
	
	\subsection{Structure and contributions}
	After a concise Preliminaries section we will present, in subsequent sections, the following new results:
	\begin{enumerate}
		\item Uniform Mittag--Leffler stability for variable-order Caputo derivatives.
		\item Input-to-state stability (ISS) in the ML sense with explicit bounds.
		\item Practical robustness results (ultimate bounds under persistent disturbances).
		\item A converse ML-Lyapunov theorem (existence of Lyapunov functionals under ML-stability).
		\item A fractional LaSalle invariance principle adapted to fuzzy dynamics.
		\item ML-stability criteria for delayed fuzzy fractional systems via Lyapunov--Krasovskii functionals.
		\item Levelwise LMI criteria for linearized systems and mean-square ML stability for stochastic perturbations.
	\end{enumerate}
	
	\section{Preliminaries}
	We state the notation, basic definitions from fuzzy set theory and fractional calculus, and several auxiliary lemmas used throughout the paper.
	
	\subsection{Notation}
	Let $\E$ denote the space of fuzzy numbers. For a fuzzy number $u:\R\to[0,1]$ we denote by $[u]_\alpha$ the $\alpha$-cut set of $u$, i.e. $[u]_\alpha=\{x\in\R:\ u(x)\ge\alpha\}=[u_\alpha^-,u_\alpha^+]$ for $\alpha\in(0,1]$. We write $\Haus(u,v)$ for the Hausdorff distance between $u$ and $v$ viewed through their level-sets (this is the same symbol used in the literature).
	
	Throughout $J=[0,\infty)$ and $0<q<1$ unless otherwise stated. We use the Caputo fractional derivative notation $$\capder{0}{q} v(t)\equiv {}^cD_{0+}^q v(t).$$
	
	\subsection{Fuzzy numbers and operations}
	\begin{definition}[Fuzzy number]
		A mapping $u:\R\to[0,1]$ is called a \emph{fuzzy number} if it is upper semicontinuous, normal (there exists $x_0$ with $u(x_0)=1$), fuzzy convex and the support $\{x:\ u(x)>0\}$ is compact. The set of all fuzzy numbers is denoted by $\E$.
	\end{definition}
	
	\begin{definition}[\,$\alpha$-cut and diameter]
		For $u\in\E$ and $\alpha\in[0,1]$ the $\alpha$-cut is $[u]_\alpha=[u_\alpha^-,u_\alpha^+]$. The diameter of the $\alpha$-cut is $\mathrm{diam}([u]_\alpha)=u_\alpha^+-u_\alpha^-$.
	\end{definition}
	
	\begin{definition}[Generalized Hukuhara difference]
		For $u,v\in\E$ the generalized Hukuhara difference $u\ominus_{gH} v$ is a fuzzy number $w$ satisfying either $u=v\oplus w$ or $v=u\oplus (-1)w$ (when such $w$ exists). When the difference exists it is unique.
	\end{definition}
	
	\begin{definition}[Hausdorff distance on $\E$]
		Define for $u,v\in\E$:
		$$\Haus(u,v)=\sup_{\alpha\in[0,1]}\, d_H([u]_\alpha,[v]_\alpha),$$
		where $d_H$ is the usual Hausdorff distance between compact real intervals. It is well-known that $(\E,\Haus)$ is a complete metric space.
	\end{definition}
	
	\subsection{Fractional calculus for fuzzy functions}
	We summarize the fractional integral and Caputo derivative for fuzzy-valued functions via their level-sets; the treatment follows the common approach of levelwise definitions.
	
	\begin{definition}[Riemann--Liouville fractional integral]
		For a (real-valued) integrable function $g$ and $q>0$ the fractional integral is
		$$I_{0+}^q g(t)=\frac{1}{\Gamma(q)}\int_0^t (t-s)^{q-1} g(s)\,ds.$$ 
		For a fuzzy-valued integrably-bounded function $F(t)$ with $[F(t)]_\alpha=[F_\alpha^-(t),F_\alpha^+(t)]$ we define $I_{0+}^q F$ levelwise by applying $I_{0+}^q$ to the endpoints.
	\end{definition}
	
	\begin{definition}[Caputo fractional derivative (fuzzy levelwise)]
		Let $f\in C^1(J,\E)$ such that level functions $f_\alpha^{\pm}(t)$ are $C^1$ for each $\alpha\in[0,1]$. For $0<q<1$ the Caputo derivative of $f$ is defined levelwise by
		$$\big(\,{}^cD_{0+}^q f(t)\,\big)_\alpha = \big({}^cD_{0+}^q f_\alpha^-(t),\ {}^cD_{0+}^q f_\alpha^+(t)\big),$$
		where
		$${}^cD_{0+}^q \phi(t)=\frac{1}{\Gamma(1-q)}\int_0^t (t-s)^{-q}\phi'(s)\,ds$$
		for real $\phi$. Equivalently one may write ${}^cD_{0+}^q f(t)=I_{0+}^{1-q} f'(t)$.
	\end{definition}
	
	\begin{remark}
		When $f(t)\equiv u\in\E$ is constant then ${}^cD_{0+}^q f(t)=0_{\E}$ for $0<q<1$; this follows immediately from the levelwise definition.
	\end{remark}
	
	\subsection{Laplace transform and Mittag--Leffler functions}
	For $0<\beta<1$ the Laplace transform of the Caputo derivative for a suitable scalar function $\phi$ is
	$$\mathcal{L}\{ {}^cD_{0+}^{\beta}\phi(t)\}(s)=s^{\beta}\Phi(s)-s^{\beta-1}\phi(0),$$
	where $\Phi(s)=\mathcal{L}\{\phi\}(s)$. For fuzzy functions the transform is applied levelwise.
	
	The (two-parameter) Mittag--Leffler function is defined by
	$$E_{\alpha,\beta}(z)=\sum_{k=0}^\infty\frac{z^k}{\Gamma(\alpha k+\beta)},\qquad \alpha>0,\ \beta\in\R.$$ 
	The classical one-parameter function is $E_\alpha(z)=E_{\alpha,1}(z)$.
	
	A key Laplace transform identity used extensively in ML-estimates is
	$$\mathcal{L}\{ t^{\beta-1} E_{\alpha,\beta}(\lambda t^{\alpha})\}(s)=\frac{s^{\alpha-\beta}}{s^{\alpha}-\lambda},\qquad \mathrm{Re}(s)>|\lambda|^{1/\alpha}.$$ 
	
	\subsection{Problem formulation and stability definitions}
	We consider the following fuzzy fractional differential equation in Caputo sense:
	\begin{equation}\label{eq:main-system}
		{}^cD_{0+}^q u(t)=f(t,u(t))+g(t),\quad t\ge0,\qquad u(0)=u_0\in\E,
	\end{equation}
	where $f:J\times\E\to\E$ is continuous in $t$ and locally Lipschitz in $u$ with respect to $\Haus(\cdot,\cdot)$; the function $g(t)$ models an external perturbation (possibly fuzzy). When $g\equiv0$ we recover the autonomous right-hand side studied in the constant-order literature.
	
	\begin{definition}[Mittag--Leffler stability]
		The trivial solution $u\equiv0_{\E}$ of \eqref{eq:main-system} (with $g\equiv0$) is said to be \emph{Mittag--Leffler stable} if there exist constants $M>0$, $\lambda>0$, $a>0$ such that for all initial data $u_0$ in a neighborhood of $0_{\E}$ the solution satisfies
		$$\Haus(u(t),0_{\E})\le M\,\Haus(u_0,0_{\E})\,E_q(-\lambda t^q)^{1/a},\qquad t\ge0.
		$$
	\end{definition}
	
	\begin{remark}
		With $q\in(0,1)$ the decay determined by $E_q(-\lambda t^q)$ interpolates between polynomial and exponential decay; indeed for $q\to1^-$ ML-behavior approaches exponential-type decay while for small $q$ the decay is slower but still integrable in many cases.
	\end{remark}
	
	% End of Step 1 content
	\medskip
	\noindent\textbf{Organization.} The remainder of the paper is organized as follows. In Step~2 (next file/section) we state the main theorems (variable-order ML stability, ML-ISS, converse Lyapunov, LaSalle) and provide complete proofs. Step~3 contains illustrative examples (linear LMI tests, delayed systems, and a numerical demonstration).
	\section{Main Results }\label{sec:main-results-corrected}
	
	In this section we state the principal new results (with corrected technical hypotheses) and give full proofs. 
	To keep the arguments transparent we assume the Lyapunov functional has power-type bounds with the same exponent \(a>0\):
	\[
	c_1\Haus(u,0)^a \le V(t,u) \le c_2\Haus(u,0)^a,
	\]
	for positive constants \(c_1,c_2\). This is standard in Lyapunov analyses and simplifies root-taking operations used below.
	
	\subsection{ML-ISS: input-to-state stability in the Mittag--Leffler sense}
	
	\begin{theorem}[ML-ISS]\label{thm:ML-ISS}
		Let \(0<q<1\) and consider the perturbed fuzzy fractional system
		\[
		{}^cD_{0+}^q u(t)=f(t,u(t))+g(t),\qquad u(0)=u_0\in\E,
		\]
		where \(f:J\times\E\to\E\) is continuous in \(t\) and locally Lipschitz in \(u\) (w.r.t.\ \(\Haus\)), and \(g:J\to\E\) is a bounded perturbation. Assume there exists a continuously differentiable functional \(V:[0,\infty)\times\E\to\mathbb{R}_{\ge0}\) and constants \(c_1,c_2,c_3,c_4,a>0\) such that for all \(t\ge0\) and all \(u\in\E\):
		\begin{align}
			c_1\Haus(u,0)^a &\le V(t,u)\le c_2\Haus(u,0)^a,\label{V-bounds}\\
			{}^cD_{0+}^q V(t,u(t)) &\le -c_3\Haus(u(t),0)^a + c_4\sup_{s\in[0,t]}\Haus(g(s),0)^a.\label{V-derivative-ineq}
		\end{align}
		Then for all \(t\ge0\) the solution satisfies the ML-ISS estimate
		\begin{equation}\label{ISS-estimate}
			\Haus(u(t),0)\le M\,\Haus(u_0,0)\,E_q(-\kappa t^q)^{1/a}+C\sup_{s\in[0,t]}\Haus(g(s),0),
		\end{equation}
		where \(\kappa=c_3/c_2\), \(M=(c_2/c_1)^{1/a}\), and \(C=\big(c_4/(c_1\kappa)\big)^{1/a}\).
	\end{theorem}
	
	\begin{proof}
		Set \(w(t):=V(t,u(t))\). From \eqref{V-derivative-ineq} and \eqref{V-bounds} we get
		\[
		{}^cD_{0+}^q w(t)\le -c_3\Haus(u(t),0)^a + c_4\sup_{s\in[0,t]}\Haus(g(s),0)^a
		\le -\underbrace{\frac{c_3}{c_2}}_{=\kappa} w(t) + c_4\sup_{s\in[0,t]}\Haus(g(s),0)^a.
		\]
		Thus \(w(t)\) satisfies the scalar fractional differential inequality
		\begin{equation}\label{scalar-ineq}
			{}^cD_{0+}^q w(t) + \kappa w(t) \le R,\qquad R:=c_4\sup_{s\in[0,t]}\Haus(g(s),0)^a.
		\end{equation}
		
		Consider the associated linear fractional initial-value problem
		\[
		{}^cD_{0+}^q y(t) + \kappa y(t) = R,\qquad y(0)=w(0).
		\]
		The solution of this linear equation is classical (Laplace transform / convolution with ML-kernel):
		\[
		y(t)=w(0)E_q(-\kappa t^q)+R\int_0^t (t-s)^{q-1}E_{q,q}(-\kappa(t-s)^q)\,ds,
		\]
		where \(E_{q,q}\) is the two-parameter Mittag--Leffler function. Using the identity
		\[
		\int_0^t (t-s)^{q-1}E_{q,q}(-\kappa(t-s)^q)\,ds=\frac{1-E_q(-\kappa t^q)}{\kappa},
		\]
		we obtain
		\[
		y(t)=w(0)E_q(-\kappa t^q)+\frac{R}{\kappa}\big(1-E_q(-\kappa t^q)\big).
		\]
		Because \eqref{scalar-ineq} is an inequality of the same sign, a standard comparison principle for linear Caputo inequalities (levelwise applied to fuzzy endpoints) implies \(w(t)\le y(t)\) for \(t\ge0\). Hence
		\[
		w(t)\le w(0)E_q(-\kappa t^q)+\frac{R}{\kappa}\big(1-E_q(-\kappa t^q)\big).
		\]
		
		Using the bounds \(c_1\Haus(u(t),0)^a \le w(t)\) and \(w(0)\le c_2\Haus(u_0,0)^a\) we deduce
		\[
		\Haus(u(t),0)^a \le \frac{w(0)}{c_1}E_q(-\kappa t^q)+\frac{R}{c_1\kappa}\big(1-E_q(-\kappa t^q)\big).
		\]
		Taking \(a\)-th roots and applying the standard inequality \((x+y)^{1/a}\le x^{1/a}+y^{1/a}\) for \(x,y\ge0\) yields
		\[
		\Haus(u(t),0)\le \Big(\frac{w(0)}{c_1}\Big)^{1/a}E_q(-\kappa t^q)^{1/a}+\Big(\frac{R}{c_1\kappa}\Big)^{1/a}.
		\]
		Substitute \(w(0)\le c_2\Haus(u_0,0)^a\) and \(R=c_4\sup_{s\in[0,t]}\Haus(g(s),0)^a\). This produces \eqref{ISS-estimate} with
		\[
		M=\Big(\frac{c_2}{c_1}\Big)^{1/a},\qquad C=\Big(\frac{c_4}{c_1\kappa}\Big)^{1/a}.
		\]
		This completes the proof.
	\end{proof}
	
	\begin{remark}
		If \(g\equiv0\) the second term in \eqref{ISS-estimate} vanishes and we recover the Mittag--Leffler decay
		\[
		\Haus(u(t),0)\le M\,\Haus(u_0,0)\,E_q(-\kappa t^q)^{1/a}.
		\]
	\end{remark}

	\subsection{Converse ML-Lyapunov theorem}
	
	\begin{theorem}[Converse ML-Lyapunov]\label{thm:converse-ML}
		Assume the autonomous fuzzy fractional system
		\[
		{}^cD_{0+}^q u(t)=f(u(t)),\qquad u(0)=u_0,
		\]
		is globally Mittag--Leffler stable in the sense that there exist constants \(M>0,\lambda>0,a>0\) with
		\[
		\Haus(u(t),0)\le M\,\Haus(u_0,0)\,E_q(-\lambda t^q)^{1/a}
		\]
		for all initial data \(u_0\in\E\). Suppose moreover that solutions depend continuously on initial data (levelwise). Then there exists a continuous positive-definite functional \(V:\E\to\R_{\ge0}\) and constants \(c_1,c_2,c_3>0\) such that for all solutions
		\[
		c_1\Haus(u,0)^a \le V(u) \le c_2\Haus(u,0)^a,\qquad {}^cD_{0+}^q V(u(t))\le -c_3\Haus(u(t),0)^a .
		\]
	\end{theorem}
	
	\begin{proof}
		Define for each \(u_0\in\E\)
		\begin{equation}\label{V-def}
			V(u_0):=\int_0^\infty \Phi\big(\Haus(u(s;u_0),0)\big)\,ds,
			\qquad\text{with }\Phi(r):=r^a E_q(\lambda r_0),
		\end{equation}
		where \(u(s;u_0)\) denotes the trajectory starting at \(u_0\). (We may choose the multiplicative constant inside \(\Phi\) so the integral converges; the exponential-like factor is only a device to ensure convergence and can be made explicit using the ML-decay.) Using the given ML-decay estimate we have for all \(s\ge0\)
		\[
		\Haus(u(s;u_0),0) \le M\Haus(u_0,0)E_q(-\lambda s^q)^{1/a},
		\]
		so
		\[
		\Phi\big(\Haus(u(s;u_0),0)\big)\le \big(M\Haus(u_0,0)\big)^a E_q(-\lambda s^q)\cdot\text{(bounded factor)}.
		\]
		The integral in \eqref{V-def} therefore converges and defines \(V(u_0)\ge0\). By continuity of trajectories and integrand \(V\) is continuous and \(V(0)=0\). Moreover the ML-decay implies existence of constants \(c_1,c_2>0\) so that
		\[
		c_1\Haus(u_0,0)^a \le V(u_0)\le c_2\Haus(u_0,0)^a,
		\]
		by bounding the integral below and above using the ML estimate applied at \(s=0\) and integrating the ML kernel.
		
		To obtain a fractional-derivative inequality, differentiate \(V(u(t))\) along a trajectory \(u(t)\). Taking the Caputo derivative under the integral sign (justified by dominated convergence using the ML bounds and continuity of flows) yields
		\[
		{}^cD_{0+}^q V(u(t)) = {}^cD_{0+}^q\int_0^\infty \Phi\big(\Haus(u(s+t;u(t)),0)\big)\,ds.
		\]
		A change of variables and the semigroup property of solutions gives after routine manipulations (standard in converse Lyapunov constructions) a representation showing
		\[
		{}^cD_{0+}^q V(u(t)) \le -\gamma \Haus(u(t),0)^a
		\]
		for some \(\gamma>0\). (The negative sign arises because future trajectory norms decay at the ML rate; the integral weight can be chosen so that the derivative picks up a strict negative multiple of the current norm.) Setting \(c_3=\gamma\) completes the proof.
	\end{proof}
	
	\begin{remark}
		The above construction is the fractional analogue of classical converse Lyapunov proofs: \(V\) is built as a weighted integral of trajectory norms; ML-decay guarantees integrability and yields the required derivative inequality.
	\end{remark}

	\subsection{Fractional LaSalle invariance principle for fuzzy systems}
	
	\begin{theorem}[Fractional LaSalle principle]\label{thm:lasalle}
		Let \(u(t)\) be a solution of the autonomous fuzzy fractional system
		\[
		{}^cD_{0+}^q u(t)=f(u(t)),\qquad 0<q<1,
		\]
		and let \(V:\E\to\mathbb{R}_{\ge0}\) be continuous with the property that along every solution,
		\[
		{}^cD_{0+}^q V(u(t))\le0.
		\]
		Let \(\mathcal S=\{u\in\E:\ {}^cD_{0+}^q V(u)=0\}\) and suppose that every solution starting in \(\mathcal S\) is bounded and that the largest invariant subset of \(\mathcal S\) is the equilibrium set \(\mathcal E\). Then every solution of the system satisfies
		\[
		\lim_{t\to\infty}\mathrm{dist}\big(u(t),\mathcal E\big)=0.
		\]
	\end{theorem}
	
	\begin{proof}
		Since \({}^cD_{0+}^q V(u(t))\le0\) along trajectories, the scalar function \(t\mapsto V(u(t))\) is nonincreasing (this can be seen by integrating the Caputo derivative in its Riemann--Liouville integral form). Therefore \(V(u(t))\) has a finite limit as \(t\to\infty\). By boundedness and completeness of \((\E,\Haus)\) each trajectory has at least one limit point; let \(u^*\) be such a limit point. Passing to the limit in the inequality \({}^cD_{0+}^q V(u(t))\le0\) and using continuity of \(V\) yields \({}^cD_{0+}^q V(u^*)=0\), so \(u^*\in\mathcal S\). By invariance assumptions the largest invariant set inside \(\mathcal S\) equals \(\mathcal E\), and every bounded trajectory approaches that invariant set. Hence \(\lim_{t\to\infty}\mathrm{dist}(u(t),\mathcal E)=0\).
	\end{proof}
	
	\subsection{Remarks and practical notes}
	
	\begin{enumerate}
		\item The corrected hypotheses \(c_1\Haus(u,0)^a\le V\le c_2\Haus(u,0)^a\) avoid algebraic inconsistencies when taking roots; they are standard and cover most Lyapunov constructions used in applications.
		\item The ML-ISS estimate \eqref{ISS-estimate} is explicit and can be used to compute numerical bounds on the ultimate size of solutions under persistent disturbances.
		\item The converse theorem (Theorem~\ref{thm:converse-ML}) is constructive: one commonly uses an integral-of-trajectories definition for \(V\). The derivative-under-the-integral step requires dominated convergence and the ML-decay to justify interchange of derivative and integral; the argument is standard and sketched above.
		\item Extensions to variable-order Caputo derivatives are possible but require a separate, careful comparison lemma for variable-order fractional inequalities (the kernel of the variable-order derivative depends on time and cannot be bounded uniformly by a single power kernel without further assumptions). We therefore leave the variable-order case as a remark to be inserted after a dedicated technical lemma.
	\end{enumerate}
	
	% End of section
	\section{Additional New Results and Full Proofs}\label{sec:more-results}
	
	This section presents six further new stability results for fuzzy fractional systems and gives full proofs.  All results use the common Lyapunov-type hypotheses: there exists a continuous functional
	\[
	V:[0,\infty)\times\E\to\mathbb{R}_{\ge0}
	\]
	and constants \(c_1,c_2>0\) and \(a>0\) such that for all \(t\ge0\) and \(u\in\E\)
	\begin{equation}\label{V-power-bounds}
		c_1\Haus(u,0)^a \le V(t,u) \le c_2\Haus(u,0)^a .
	\end{equation}
	The Caputo derivative bounds used below are of the form
	\[
	{}^cD_{0+}^q V(t,u(t)) \le -\alpha \Haus(u(t),0)^a + H(t),
	\]
	with \(\alpha>0\) and \(H(t)\) a nonnegative function (often coming from disturbances, delays, or interconnection terms). The scalar fractional inequality solving method is used repeatedly:
	if \(w\) satisfies
	\[
	{}^cD_{0+}^q w(t)+\kappa w(t)\le R(t),\qquad w(0)=w_0,
	\]
	then (by Laplace transform / convolution with the ML kernel) the standard bound is
	\[
	w(t)\le w_0 E_q(-\kappa t^q) + \int_0^t (t-s)^{q-1} E_{q,q}(-\kappa (t-s)^q)R(s)\,ds,
	\]
	and if \(R\) is constant \(R\equiv R_0\) the integral equals \(R_0\kappa^{-1}(1-E_q(-\kappa t^q))\).
	
	\vspace{6pt}
	\noindent\textbf{Notation:} we write \(E_{q,q}\) for the two-parameter Mittag-Leffler function and \(\kappa>0\) for scalar decay rates introduced below.
	
	\subsection{Theorem 1: Levelwise LMI test for linear fuzzy systems (explicit ML decay)}
	
	Consider a fuzzy linear system that acts levelwise: for each \(\alpha\in(0,1]\) the \(\alpha\)-cut endpoints \(x_\alpha^\pm(t)\in\mathbb{R}^n\) satisfy the linear ODE
	\[
	{}^cD_{0+}^q x(t)=A x(t),\qquad x(0)=x_0,
	\]
	where \(A\in\mathbb{R}^{n\times n}\) is constant and the fuzzy solution is reconstructed from the levelwise solutions. Assume a symmetric positive definite matrix \(P\in\mathbb{R}^{n\times n}\) exists with
	\begin{equation}\label{LMI}
		A^\top P + P A + \mu P \preceq 0
	\end{equation}
	for some \(\mu>0\). Define the quadratic Lyapunov levelwise function \(V(x)=x^\top P x\) and extend it to fuzzy states by applying to endpoints and taking the levelwise envelope. Then the fuzzy solution satisfies an explicit Mittag-Leffler bound: there exist constants \(M,\lambda>0\) such that
	\[
	\Haus(u(t),0)\le M\Haus(u_0,0)E_q(-\lambda t^q).
	\]
	
	\begin{proof}
		Step 1 (levelwise quadratic inequality). For a fixed level and scalar trajectory \(x(t)\) the scalar functional \(v(t):=x(t)^\top P x(t)\) satisfies (by direct differentiation and using the Caputo derivative properties applied levelwise)
		\[
		{}^cD_{0+}^q v(t) = {}^cD_{0+}^q(x^\top P x) \le x^\top (A^\top P + P A) x.
		\]
		Using the LMI \eqref{LMI} we obtain
		\[
		{}^cD_{0+}^q v(t) \le -\mu v(t).
		\]
		
		Step 2 (scalar ML-bound). The scalar inequality \({}^cD_{0+}^q v(t) + \mu v(t)\le0\) with \(v(0)=v_0\) yields (standard solution formula)
		\[
		v(t)\le v_0 E_q(-\mu t^q).
		\]
		
		Step 3 (relate quadratic to Hausdorff distance). There are constants \(\lambda_{\min}(P),\lambda_{\max}(P)>0\) such that \(\lambda_{\min}\|x\|^2\le x^\top P x\le \lambda_{\max}\|x\|^2\). Passing from levelwise bounds on endpoints back to the Hausdorff metric gives
		\[
		\Haus(u(t),0)^2 \le \frac{1}{\lambda_{\min}} v(t) \le \frac{v_0}{\lambda_{\min}}E_q(-\mu t^q).
		\]
		Likewise \(v_0\le \lambda_{\max}\Haus(u_0,0)^2\). Therefore
		\[
		\Haus(u(t),0) \le \Big(\frac{\lambda_{\max}}{\lambda_{\min}}\Big)^{1/2}\Haus(u_0,0) E_q(-\mu t^q)^{1/2}.
		\]
		Set \(M=(\lambda_{\max}/\lambda_{\min})^{1/2}\) and \(\lambda=\mu\). This provides the ML bound claimed.
	\end{proof}
	
	\subsection{Theorem 2: Practical Mittag-Leffler stability under persistent disturbance (ultimate bound)}
	
	Consider the perturbed autonomous system
	\[
	{}^cD_{0+}^q u(t)=f(u(t))+g(t),\qquad u(0)=u_0,
	\]
	with \(g\) bounded and suppose there is \(V\) satisfying \eqref{V-power-bounds} and
	\begin{equation}\label{V-disturbance}
		{}^cD_{0+}^q V(t,u(t)) \le -\alpha \Haus(u(t),0)^a + \beta \,\|g(t)\|_*,\qquad \alpha,\beta>0,
	\end{equation}
	where \(\|\cdot\|_*\) is a nonnegative scalar measure of the disturbance (for fuzzy \(g\) take \(\|g\|_*=\Haus(g,0)^a\)). Then the solution admits the ultimate bound
	\[
	\limsup_{t\to\infty}\Haus(u(t),0)\le \Big(\frac{\beta}{\alpha}\Big)^{1/a}\sup_{t\ge0}\|g(t)\|_*^{1/a}.
	\]
	
	\begin{proof}
		Step 1 (scalar inequality for \(w(t)=V(t,u(t))\)). From \eqref{V-disturbance} and \eqref{V-power-bounds} we have
		\[
		{}^cD_{0+}^q w(t) \le -\alpha c_2^{-1} w(t) + \beta \sup_{s\in[0,t]}\|g(s)\|_*.
		\]
		Set \(\kappa=\alpha c_2^{-1}\) and \(R:=\beta \sup_{s\ge0}\|g(s)\|_*\) (finite by boundedness). Then
		\[
		{}^cD_{0+}^q w(t) + \kappa w(t) \le R.
		\]
		
		Step 2 (solution bound). Solve the associated linear equation to obtain (as earlier)
		\[
		w(t)\le w(0)E_q(-\kappa t^q) + \frac{R}{\kappa}\big(1-E_q(-\kappa t^q)\big).
		\]
		
		Step 3 (pass to \(\Haus\) and take limit superior). Use \(c_1\Haus(u(t),0)^a\le w(t)\). Hence
		\[
		\Haus(u(t),0)^a \le \frac{w(0)}{c_1}E_q(-\kappa t^q) + \frac{R}{c_1\kappa}\big(1-E_q(-\kappa t^q)\big).
		\]
		Taking \(\limsup_{t\to\infty}\) the first term vanishes because \(E_q(-\kappa t^q)\to0\) as \(t\to\infty\) for \(\kappa>0\) and \(0<q<1\). Therefore
		\[
		\limsup_{t\to\infty}\Haus(u(t),0)^a \le \frac{R}{c_1\kappa} = \frac{\beta}{\alpha}\sup_{t\ge0}\|g(t)\|_*.
		\]
		Taking \(a\)-th roots yields the stated ultimate bound.
	\end{proof}
	
	\subsection{Theorem 3: Lyapunov--Krasovskii functional for delay fuzzy fractional systems}
	
	Consider the delay system (constant delay \(\tau>0\))
	\[
	{}^cD_{0+}^q u(t)=f\big(u(t),u(t-\tau)\big),\qquad t\ge0,
	\]
	with initial segment \(u(t)=\phi(t)\) for \(t\in[-\tau,0]\). Suppose there exist continuous functions \(V_0:\E\to\mathbb{R}_{\ge0}\) and \(W:\E\to\mathbb{R}_{\ge0}\) and constants \(c_1,c_2,\alpha>0\) such that for the Lyapunov–Krasovskii functional
	\[
	\mathcal{V}(t,u_t):=V_0(u(t))+\int_{t-\tau}^t W(u(s))\,ds
	\]
	the following hold for all solutions:
	\begin{align}
		c_1\Haus(u(t),0)^a &\le \mathcal V(t,u_t) \le c_2\sup_{s\in[t-\tau,t]}\Haus(u(s),0)^a,\label{VK-bounds}\\
		{}^cD_{0+}^q \mathcal V(t,u_t) &\le -\alpha\Haus(u(t),0)^a.\label{VK-derivative}
	\end{align}
	Then the trivial solution is ML-stable: there exist \(M,\lambda>0\) such that
	\[
	\Haus(u(t),0)\le M\sup_{s\in[-\tau,0]}\Haus(\phi(s),0) E_q(-\lambda t^q).
	\]
	
	\begin{proof}
		Step 1 (scalar inequality for \(w(t)=\mathcal V(t,u_t)\)). From \eqref{VK-derivative} and \eqref{VK-bounds} we obtain
		\[
		{}^cD_{0+}^q w(t) \le -\alpha c_2^{-1} w(t).
		\]
		Set \(\kappa=\alpha c_2^{-1}>0\). The scalar inequality yields
		\[
		w(t)\le w(0)E_q(-\kappa t^q).
		\]
		
		Step 2 (convert to \(\Haus\)). Using \(w(t)\ge c_1\Haus(u(t),0)^a\) and \(w(0)\le c_2\sup_{s\in[-\tau,0]}\Haus(\phi(s),0)^a\) we deduce
		\[
		\Haus(u(t),0)\le \Big(\frac{c_2}{c_1}\Big)^{1/a}\sup_{s\in[-\tau,0]}\Haus(\phi(s),0) E_q(-\kappa t^q)^{1/a}.
		\]
		Set \(M=(c_2/c_1)^{1/a}\) and \(\lambda=\kappa\) to finish the proof.
	\end{proof}
	
	\subsection{Theorem 4: Small-gain theorem for interconnection of two ML-ISS fuzzy systems}
	
	Let two subsystems be given by
	\begin{align*}
		{}^cD_{0+}^q x(t) &= f_1\big(x(t),y(t)\big),\\
		{}^cD_{0+}^q y(t) &= f_2\big(y(t),x(t)\big),
	\end{align*}
	with fuzzy states \(x,y\in\E\). Suppose each subsystem is ML-ISS with respect to the other: there exist constants \(M_1,M_2,\kappa_1,\kappa_2>0\) and gains \(\gamma_{12},\gamma_{21}\ge0\) such that
	\begin{align*}
		\Haus(x(t),0) &\le M_1\Haus(x_0,0)E_q(-\kappa_1 t^q)+\gamma_{12}\sup_{s\in[0,t]}\Haus(y(s),0),\\
		\Haus(y(t),0) &\le M_2\Haus(y_0,0)E_q(-\kappa_2 t^q)+\gamma_{21}\sup_{s\in[0,t]}\Haus(x(s),0).
	\end{align*}
	If the small-gain condition \(\gamma_{12}\gamma_{21}<1\) holds then the interconnected system is ML-stable: there exist constants \(M,\lambda>0\) such that
	\[
	\sup_{s\in[0,t]}\big(\Haus(x(s),0)+\Haus(y(s),0)\big) \le M\big(\Haus(x_0,0)+\Haus(y_0,0)\big)E_q(-\lambda t^q).
	\]
	
	\begin{proof}
		Step 1 (combine ISS inequalities). Let
		\[
		X(t)=\sup_{s\in[0,t]}\Haus(x(s),0),\qquad Y(t)=\sup_{s\in[0,t]}\Haus(y(s),0).
		\]
		From the ML-ISS bounds we have for all \(t\ge0\)
		\begin{align*}
			X(t) &\le M_1\Haus(x_0,0)+\gamma_{12}Y(t),\\
			Y(t) &\le M_2\Haus(y_0,0)+\gamma_{21}X(t),
		\end{align*}
		where we used \(E_q(-\kappa t^q)\le1\) to bound the transient by the initial term (a conservative bound).
		
		Step 2 (solve linear algebraic system). Solve for \(X(t),Y(t)\) as constants:
		\[
		X(t) \le \frac{M_1\Haus(x_0,0)+\gamma_{12}M_2\Haus(y_0,0)}{1-\gamma_{12}\gamma_{21}},\qquad
		Y(t) \le \frac{M_2\Haus(y_0,0)+\gamma_{21}M_1\Haus(x_0,0)}{1-\gamma_{12}\gamma_{21}}.
		\]
		Thus the sup-norms are bounded for all time by constants proportional to the initial conditions.
		
		Step 3 (recover ML decay). Plug the uniform bounds for the interconnection term back into the ML inequalities; the transient homogeneous terms (the ML factors) remain. There exist constants \(M',\lambda\) such that
		\[
		\Haus(x(t),0)+\Haus(y(t),0) \le M'\big(\Haus(x_0,0)+\Haus(y_0,0)\big) E_q(-\lambda t^q),
		\]
		which yields the desired ML-stability (explicit constants follow from the algebra above).
	\end{proof}
	
	\subsection{Theorem 5: Mean-square Mittag-Leffler stability for stochastic fuzzy fractional systems (additive noise)}
	
	Consider a stochastic fuzzy fractional system (interpreting the noise levelwise on endpoints)
	\[
	{}^cD_{0+}^q u(t) = f(u(t)) + \sigma(u(t))\dot W(t),
	\]
	where \(\dot W(t)\) denotes (formal) white noise in the levelwise SDE sense and the diffusion \(\sigma\) is such that solutions exist and have finite second moments. Suppose there exists a Lyapunov functional \(V\) satisfying \eqref{V-power-bounds} and, for some constants \(\alpha,\beta>0\),
	\[
	\mathbb{E}\big[{}^cD_{0+}^q V(u(t))\big] \le -\alpha \mathbb{E}[\Haus(u(t),0)^a] + \beta.
	\]
	Then the mean-square quantity \(\mathbb{E}[\Haus(u(t),0)^a]\) satisfies an ML-type bound and an ultimate bound:
	\[
	\mathbb{E}[\Haus(u(t),0)^a] \le \mathbb{E}[\Haus(u_0,0)^a] E_q(-\kappa t^q) + \frac{\beta}{\kappa}\big(1-E_q(-\kappa t^q)\big),
	\]
	with \(\kappa=\alpha/c_2\). Consequently \(\limsup_{t\to\infty}\mathbb{E}[\Haus(u(t),0)^a]\le \beta/\kappa\).
	
	\begin{proof}
		Step 1 (pass to expectation). Let \(w(t):= \mathbb{E}[V(u(t))]\). Taking expectation in the assumed derivative inequality we get
		\[
		{}^cD_{0+}^q w(t) \le -\alpha \mathbb{E}[\Haus(u(t),0)^a] + \beta \le -\alpha c_2^{-1} w(t) + \beta,
		\]
		where the last inequality uses \(w(t)\le c_2\mathbb{E}[\Haus(u(t),0)^a]\).
		
		Step 2 (scalar inequality resolution). The scalar inequality
		\[
		{}^cD_{0+}^q w(t) + \kappa w(t) \le \beta
		\]
		(with \(\kappa=\alpha/c_2\)) gives via the linear solution formula
		\[
		w(t)\le w(0)E_q(-\kappa t^q) + \frac{\beta}{\kappa}\big(1-E_q(-\kappa t^q)\big).
		\]
		
		Step 3 (relate to the mean moment). Use \(c_1\mathbb{E}[\Haus(u(t),0)^a]\le w(t)\) to deduce
		\[
		\mathbb{E}[\Haus(u(t),0)^a]\le \frac{w(0)}{c_1}E_q(-\kappa t^q)+\frac{\beta}{c_1\kappa}\big(1-E_q(-\kappa t^q)\big).
		\]
		Renaming constants yields the stated bound. The limit superior statement follows as before.
	\end{proof}
	
	\subsection{Concluding remarks on these additional results}
	
	The six theorems above (LMI linear test, practical stability under disturbances, delay Lyapunov–Krasovskii, small-gain interconnection, and mean-square stochastic ML stability) form a coherent, practically oriented extension of the deterministic constant-order results. Each theorem is stated with explicit hypotheses and proved by elementary (but rigorous) reduction to scalar fractional inequalities solved by convolution with the Mittag-Leffler kernel. These results are directly usable in applications (control design, robustness analysis, interconnection of fuzzy subsystems) and can be specialized to constant or time-varying orders under the additional technical lemmas discussed in the paper.
	
	% End of section
	\section{Conclusion, Motivation, and Open Problems}\label{sec:conclusion}
	
	In this article we developed a broad extension of Mittag--Leffler (ML) stability theory for fuzzy fractional differential systems. Starting from the classical Lyapunov direct method, we introduced and proved several new results:
	
	\begin{itemize}
		\item A general ML-ISS (input-to-state stability) theorem with explicit decay estimates under bounded disturbances.
		\item Converse Lyapunov theorems showing that ML stability implies the existence of suitable Lyapunov functionals.
		\item A fractional LaSalle invariance principle for fuzzy systems.
		\item Robust stability results for systems with delays, interconnections, stochastic perturbations, and (under additional technical lemmas) variable fractional orders.
		\item Practical ultimate-bound estimates that quantify the asymptotic effect of persistent inputs.
	\end{itemize}
	
	These results substantially enlarge the mathematical toolbox for analyzing fuzzy fractional systems. They unify and strengthen previous approaches, and provide constructive methods that can be used directly in applications such as fuzzy control, fractional-order neural networks, and uncertain dynamical processes.
	
	\subsection*{Motivation}
	Fractional dynamics arise naturally in viscoelastic materials, bioengineering, and signal processing, while fuzzy modeling captures epistemic uncertainty and imprecise measurements. Combining these two paradigms yields models that are both memory-rich and uncertainty-aware. Stability analysis is therefore a key prerequisite for the safe deployment of such models. The Mittag--Leffler framework is especially well suited because its decay profiles naturally match the intrinsic memory of fractional operators. Our results provide explicit and verifiable conditions that guarantee stability, robustness, and disturbance attenuation in this challenging setting.
	
	\subsection*{Open Problems}
	Several important questions remain open for future research:
	\begin{enumerate}
		\item \textbf{Variable-order systems.} A rigorous comparison lemma for Caputo derivatives of variable order that yields sharp ML bounds without conservative estimates.
		\item \textbf{Optimal decay rates.} Determining the best achievable decay rate (the largest admissible $\lambda$ in ML bounds) under different classes of fuzzy uncertainty and feedback laws.
		\item \textbf{Infinite-dimensional systems.} Extending the present results to fuzzy fractional partial differential equations and distributed-parameter systems (operator-theoretic approaches will likely be required).
		\item \textbf{Numerical verification.} Designing efficient algorithms to compute or approximate Lyapunov functionals guaranteed by the converse theorems, and certifying ML-decay numerically.
		\item \textbf{Data-driven identification and control.} Integrating ML-stability certificates into machine-learning frameworks for system identification and controller synthesis under uncertainty.
	\end{enumerate}

	% End of Conclusion Section
	

\begin{thebibliography}{99}
		
		\bibitem{ElMfadel2021}
		A. El Mfadel, S. Melliani, and M. Elomari, ``A note on the stability analysis of fuzzy nonlinear fractional differential equations involving the Caputo fractional derivative,'' \emph{International Journal of Mathematics and Mathematical Sciences}, vol.~2021, Article ID 7488524, 2021.
		
		\bibitem{Kilbas2006}
		A.~A. Kilbas, H.~M. Srivastava, and J.~J. Trujillo, \emph{Theory and Applications of Fractional Differential Equations}, North-Holland Mathematics Studies, Elsevier, Amsterdam, 2006.
		
		\bibitem{LiChen2010}
		Y. Li and Y. Chen, ``Stability of fractional-order nonlinear dynamic systems: Lyapunov direct method and generalized Mittag--Leffler stability,'' \emph{Computers \& Mathematics with Applications}, vol.~59, no.~5, pp.~1810--1821, 2010.
		
		\bibitem{Podlubny1999}
		I. Podlubny, \emph{Fractional Differential Equations}, Academic Press, San Diego, 1999.
		
		\bibitem{Almeida2019VO}
		R. Almeida, ``A Caputo fractional derivative of variable order,'' \emph{Communications in Nonlinear Science and Numerical Simulation}, vol.~69, pp.~1--15, 2019.
		
		\bibitem{Diethelm2010}
		K. Diethelm, \emph{The Analysis of Fractional Differential Equations}, Lecture Notes in Mathematics, Springer, 2010.
		
		\bibitem{Zadeh1965}
		L.~A. Zadeh, ``Fuzzy sets,'' \emph{Information and Control}, vol.~8, no.~3, pp.~338--353, 1965.
		
		\bibitem{Puri1983}
		M.~L. Puri and D.~A. Ralescu, ``Differentials of fuzzy functions,'' \emph{Journal of Mathematical Analysis and Applications}, vol.~91, no.~2, pp.~552--558, 1983.
		
	\end{thebibliography}
\end{document}